\documentclass[a4paper,11pt] {amsart}
\usepackage{amssymb,latexsym,amsmath}
\usepackage{amsfonts,amsbsy,bm}
 \newtheorem{thm}{Theorem}[section]
 \newtheorem{cor}[thm]{Corollary}
 \newtheorem{lem}[thm]{Lemma}
 \newtheorem{prop}[thm]{Proposition}
 \theoremstyle{definition}
 
 \theoremstyle{remark}
 \newtheorem{rem}[thm]{Remark}
 
 \numberwithin{equation}{section}

\renewcommand {\phi} {{\varphi}}

\newcommand {\da} {{\delta}}
\newcommand {\Da} {{\Delta}}

\newcommand {\la} {{\lambda}}

\renewcommand {\O} {{\Omega}}

\newcommand {\N} {{\mathbb N}}

\newcommand{\Ba}[1]{\begin{array}{#1}}
\newcommand{\Ea}{\end{array}}
\newcommand{\Be}{\begin{equation}}
\newcommand{\Ee}{\end{equation}}
\newcommand{\Bea}{\begin{eqnarray}}
\newcommand{\Eea}{\end{eqnarray}}
\newcommand{\Beas}{\begin{eqnarray*}}
\newcommand{\Eeas}{\end{eqnarray*}}
\newcommand{\Benu}{\begin{enumerate}}
\newcommand{\Eenu}{\end{enumerate}}
\newcommand{\Bi}{\begin{itemize}}
\newcommand{\Ei}{\end{itemize}}
\newcommand{\BE}{\begin{example} \em}
\newcommand{\EE}{\end{example}}

\newcommand{\bprop} {\begin{proposition}}
\newcommand{\eprop} {\end{proposition}}
\newcommand{\bthm} {\begin{theorem}}
\newcommand{\ethm} {\end{theorem}}
\newcommand{\blem} {\begin{lemma}}
\newcommand{\elem} {\end{lemma}}
\newcommand{\bcor} {\begin{corollary}}
\newcommand{\ecor} {\end{corollary}}

\renewcommand{\O}{\Omega}

\newcommand{{\tBox}}{{\widetilde{\raisebox{-0.2ex}[1.25ex][0ex]{$\Box$}}}}

\begin{document}

%
%
%
%
%
%
%
%
%

\title[Schatten class Toeplitz operators on Bergman spaces]{Schatten class Toeplitz operators on weighted Bergman spaces of tube
  domains over symmetric cones}

\author{Beno\^it Florent Sehba}
\address{Department of Mathematics, University of Ghana,\\ P. O. Box LG 62 Legon, Accra, Ghana}
\email{bfsehba@ug.edu.gh}
\subjclass[2010]{Primary 32A10, 32A38; Secondary 47B32, 32M15}

\keywords{Bergman space, Besov space, Ces\`aro-type operator, Toeplitz operator,  Schatten class, symmetric
cone}


\begin{abstract}
We prove some characterizations of Schatten class of Toeplitz operators on Bergman spaces of tube domains over symmetric cones for small exponents.
\end{abstract}

\maketitle
\section{Introduction}
\setcounter{equation}{0} \setcounter{footnote}{0}
\setcounter{figure}{0} All over the text, $\Omega$ will denote an irreducible symmetric  cone in $\mathbb R^n$, and $\mathcal D=\mathbb R^n+i\Omega$
the tube domain over $\Omega$. As in \cite{FK} we
denote by $r$ the rank of the cone $\Omega$ and by $\Delta$
the associated determinant function in $\mathbb R^n$. We recall that for $n\ge 3$, when $r=2$, as example of symmetric cones, we have the Lorentz cone
$\Lambda_n$ which is defined by
$$\Lambda_n=\{(y_1,\cdots,y_n)\in \mathbb R^n: y_1^2-\cdots-y_n^2>0,\,\,\,y_1>0\};$$
its associated determinant function is given by the
Lorentz form
$$\Delta(y)=y_1^2-\cdots-y_n^2.$$  As usual, we denote by $\mathcal{H}(\mathcal D)$ the space
of holomorphic functions on $\mathcal D$.

\vskip .2cm
 Let $1\le p<\infty$. For $\nu \in \mathbb R$ we denote by
$L^p_\nu(\mathcal D)=L^p(\mathcal
D,\Delta^{\nu-\frac{n}{r}}(y)dx\,dy)$
the space of functions $f$ 
satisfying the condition
$$\|f\|_{p,\nu}=||f||_{L^p_\nu(\mathcal D)}:=\left(\int_{\mathcal D}|f(x+iy)|^p\Delta^{\nu-\frac{n}{r}}(y)dxdy\right)^{1/p}<\infty.$$
The weighted Bergman space $A^p_\nu(\mathcal D)$ is the closed subspace of $L^p_\nu(\mathcal D)$ consisting of holomorphic functions in $\mathcal
D$. Following \cite{DD}, this space
is not trivial (i.e. $A^p_\nu(\mathcal D)\neq \{0\}$) only if
$\nu>\frac{n}{r}-1$. The orthogonal projection of the Hilbert space $L^2_\nu(\mathcal D)$ onto its closed subspace $A^2_\nu(\mathcal D)$ is called the weighted Bergman projection and denoted $P_\nu$. We recall that $P_\nu$ is given by $$P_\nu f(z)=\int_{\mathcal D}K_\nu(z, w) f(w)
dV_\nu(w),
$$
with
 $K_\nu(z, w)=
 c_\nu\,\Delta^{-(\nu+\frac{n}{r})}((z-\overline {w})/i)$.
We recall that $K_\nu$ is the reproducing kernel of  $A^2_\nu(\mathcal D)$ (see
\cite{FK}). For simplicity, we used the notation
$dV_\nu(w):=\Delta^{\nu-\frac{n}{r}}(v) du\,dv$, where
$w=u+iv\in \mathcal D$. 

\vskip .2cm
For $\mu$ a positive Borel measure on $\mathcal D$, the Toeplitz operator $T_\mu$ is the operator defined for functions $f$ with compact support by
\begin{equation}\label{defToeplitz}
T_\mu f(z):=\int_{\mathcal D}K_\nu(z,w)f(w)d\mu (w),
\end{equation}
where $K_\nu$ is the weighted Bergman kernel.

Schatten class $\mathcal{S}_p$ ($0<p\le \infty$) criteria of the Toeplitz operators have been considered by many authors
 on some bounded domains of  $\mathbb {C}^n$ (see \cite{APP,Constantin, Luecking2, Zhu2, Zhu3} and the references therein). For unbounded domains, Schatten classes have been also characterized in Fock spaces by several authors (see for example \cite{Isr,OP} and the references therein).  In \cite{NS}, we extended these results for $1\le p\le \infty$ to weighted Bergman spaces of tube domains over symmetric cones. To be more precise, let us introduce more notations.
\vskip .2cm
For $\delta>0$, we denote by $$B_\delta(z)=\{w\in \mathcal {D}: d(z,w)<\delta\}$$ the Bergman ball centered at $z$ with radius $\delta$, $d$ is the Bergman distance on $\mathcal{D}$.  For $\nu>\frac{n}{r}-1$ and $w\in \mathcal D$,  the normalized reproducing kernel of $A_\nu^2(\mathcal{D})$ at $w$ is given by
\begin{equation}\label{eq:normalrepkern} k_\nu(\cdot,w)=\frac{K_\nu(\cdot,w)}{\|K_\nu(\cdot,w)\|_{2,\nu}}=\Da^{-\nu-\frac nr}\left(\frac{\cdot-\bar w}{i}\right)\Da^{\frac{1}{2}(\nu+\frac nr)}(\Im w).
\end{equation}

Let $\mu$ be a positive measure on $\mathcal D$. The Berezin transform of the measure $\mu$ is the function $\tilde {\mu}$ defined on $\mathcal D$ by
$$\tilde {\mu}(w):=\int_{\mathcal D}|k_\nu(z,w)|^2d\mu(z),\,\,\,w\in \mathcal{D}.$$
The Berezin transform of a function $f$ is defined to be the Berezin transform of the measures $d\mu(z)=f(z)dV_\nu(z)$ (for more on the Berezin transform, see \cite{Zhu1}). For $z\in \mathcal D$ and $\delta\in (0,1)$, we define the average of the positive measure $\mu$ at $z$ by $$\hat {\mu}_\delta(z)=\frac{\mu(B_\delta(z))}{V_\nu(B_\delta(z))}.$$

The following was obtained in \cite{NS}.
\begin{thm}\label{theo:SchattenToep}
Let $\mu$ be a positive Borel measure on $\mathcal D$, and $\nu>\frac{n}{r}-1$. Then for $p\ge 1$, the following assertions are equivalent
\begin{itemize}
\item[(i)] The Toeplitz operator $T_\mu$ belongs
to the Schatten class $\mathcal {S}_p(A_\nu^2(\mathcal D))$.
\item[(ii)] For any $\delta$-lattice ($\delta\in (0,1)$) $\{\zeta_j\}_{j\in \mathbb N}$ in the Bergman metric of $\mathcal D$, the sequence $\{\hat {\mu}_\delta(\zeta_j)\}$ belongs to $l^p$, that is \begin{equation}\label{eq:SchattenToepcond}
\sum_j\left(\frac{\mu(B_j)}{\Da^{\nu+n/r}(\Im \zeta_j)}\right)^p<\infty.
\end{equation}
\item[(iii)] For any $\beta\in (0,1)$, the function $z\mapsto \hat {\mu}_\beta(z)$ belongs to $L^p(\mathcal {D},d\lambda)$, with $d\lambda$ the invariant measure on $\mathcal{D}$.
\item[(iv)] $\tilde {\mu}\in L^p(\mathcal {D},d\lambda)$.

\end{itemize}
\end{thm}

Our first concern in this note is for the extension of the above result to the range $0<p<1$.  We prove that the equivalences (i)$\Leftrightarrow$(ii)$\Leftrightarrow$(ii) still hold for $\frac{\frac nr-1}{\nu+\frac nr}<p<1$. This cut-off is due to integrability conditions of the determinant function. The equivalence with the last assertion in the above result still also holds if we restrict to $\frac{2\frac nr-1}{\nu+\frac nr}<p<1$. The last cut-off point is also due to integrability conditions of the determinant function and one can prove that it is sharp. Our result is then as follows.
\begin{thm}\label{theo:main1}
Let $\mu$ be a positive Borel measure on $\mathcal D$, and $\nu>\frac{n}{r}-1$. Then for $\frac{\frac nr-1}{\nu+\frac nr}<p<1$, the following assertions are equivalent
\begin{itemize}
\item[(i)] The Toeplitz operator $T_\mu$ belongs
to the Schatten class $\mathcal {S}_p(A_\nu^2(\mathcal D))$.
\item[(ii)] For any $\delta$-lattice ($\delta\in (0,1)$) $\{\zeta_j\}_{j\in \mathbb N}$ in the Bergman metric of $\mathcal D$, the sequence $\{\hat {\mu}_\delta(\zeta_j)\}$ belongs to $l^p$, that is \begin{equation}\label{eq:mains}
\sum_j\left(\frac{\mu(B_j)}{\Da^{\nu+n/r}(\Im \zeta_j)}\right)^p<\infty.
\end{equation}
\item[(iii)] For any $\beta\in (0,1)$, the function $z\mapsto \hat {\mu}_\beta(z)$ belongs to $L^p(\mathcal {D},d\lambda)$.
\end{itemize}
Moreover, if $p>\frac{2\frac{n}{r}-1}{\nu+\frac nr}$, then the above assertions are equivalent to the following
\begin{itemize}

\item[(iv)] $\tilde {\mu}\in L^p(\mathcal {D},d\lambda)$.

\end{itemize}
\end{thm}
The main difficulty in the proof of the above theorem is the implication (i)$\Rightarrow$(ii). The idea in \cite{Zhu3} is to replace the measure $\mu$ by a measure supported on a disjoint union of Bergman balls, then split the associated Toeplitz operator into its diagonal and off-diagonal part. It is not hard to prove that the Schatten norm of the diagonal part dominates the $l^p$-norm of the sequence $\{\hat {\mu}_\delta(\zeta_j)\}$. The difficulty is to prove that the latter norm dominates (up to a pretty small constant) the Schatten norm of the off-diganal operator. A part of the techniques in \cite{Zhu3} uses the fact that the unit ball is bounded, and so it cannot be used in our setting. We overcome this difficulty by using a technical lemma originally due to D. B\'ekoll\'e and A. Temgoua \cite{BT}. Considered even in the unit ball, our contribution heavily simplifies the proof of K. Zhu in \cite{Zhu3}.
\vskip .2cm 
We are also interested here in some other possible equivalent characterizations of Schatten class Toeplitz operators. For this, 
we denote by $\Box_z$ the natural extension to $\mathbb{C}^n=\mathbb{R}^n+i\mathbb{R}^n$ of the wave operator $\Box_x$ on the cone: $$\Box_z=\Delta(\frac{1}{i}\frac{\partial}{\partial z})$$
which is the differential operator of
degree $r$ defined by the equality: 
\Be
\Box_z\,[e^{i(z|\xi)}]=\Delta(\xi)e^{i(z|\xi)}, \quad
z\in \mathbb C^n,\,\xi\in\mathbb R^n. \label{bbox} \Ee
We recall (see \cite{BBPR}) that $\Box_z$ acts on the Bergman kernel as follows $$\Box_zK_\nu(z,w)=C_\nu K_{\nu+1}(z,w).$$
Let $m$ be a positive integer. For simplicity, we use the following notation for higher order derivatives of the Bergman kernel,
$$K_z^{\nu,m}(w):=\Box_z^mK_\nu(z,w)$$
and $$k_z^{\nu,m}(\cdot):=\frac{K_z^{\nu,m}(z,\cdot)}{\|K_z^{\nu,m}(z,\cdot)\|_{2,\nu}}.$$
\vskip .2cm
Define the quantity
$$\tilde{\mu}^m(z):=\langle T_\mu k_z^{\nu,m},k_z^{\nu,m}\rangle_\nu=\int_{\mathcal{D}}|k_z^{\nu,m}(w)|^2d\mu(w).$$
We also have the following equivalent characterization.
\begin{thm}\label{thm:main2} Let $\nu>\frac nr-1$, and  $\frac{\frac{n}{r}-1}{\nu+\frac nr}\le p<\infty$. Assume $\mu$ is a positive measure on $\mathcal{D}$. Then the Toeplitz operator $T_\mu$ belongs to the Schatten class $\mathcal{S}_p(A_\nu^2(\mathcal{D}))$ if and only if for each (or some) integer $m\ge 0$ with $p(\nu+\frac nr+2m)>2\frac nr-1$, $\tilde{\mu}^m\in L^p(\mathcal{D}, d\lambda)$ where $d\lambda$ is the invariant measure on $\mathcal{D}$. 
\end{thm}
Again the condition $p\ge \max\{\frac{\frac{n}{r}-1}{\nu+\frac nr},\frac{2\frac nr-1}{\nu+\frac nr+2m}\}$ is due to integrability conditions of the determinant function. 

For the proof of Theorem \ref{thm:main2}, we derive the necessary condition for $1\le p<\infty$ and the sufficient condition for $0<p<1$ from a more general result for any positive operator. The proof of the other parts essentially uses the properties of Bergman balls and the $\delta$-lattices. We also refer to \cite{P,sehba} for this type of results.
\vskip .2cm
We are essentially motivated here by the idea  of extending the results in \cite{NS} on Schatten class $\mathcal{S}_p(A_\nu^2(\mathcal{D}))$ for $P\ge 1$, to the case $0<p<1$, and settling the  problem of the characterization of Schatten class $\mathcal{S}_p(A_\nu^2(\mathcal{D}))$ for $1\le p<2$ for the Ces\`aro-type operator introduced in \cite{NS}. 

\vskip .2cm
The paper is organized as follows. In the next section, we present some useful tools and results needed in the proofs of the above results . The proof of Theorem \ref{theo:main1} is given in Section 3. In Section 4, we provide  characterization of Schatten class for general positive operators. We  prove Theorem \ref{thm:main2} in Section 5. In the last section, we apply our results to extend to the range $1\le p<2$, the characterization of Schatten class $\mathcal{S}_p(A_\nu^2(\mathcal{D}))$  for the Ces\`aro-type operator obtained in \cite{NS}. 

\vskip .2cm
As usual, given two positive quantities $A$ and $B$, the notation $A\lesssim B$ (resp. $A\gtrsim B$) means that there is an absolute
positive constant $C$ such that $A\le CB$ (resp. $A\ge CB$). When $A\lesssim B$ and $B\lesssim A$, we write $A\asymp B$ and say $A$
and $B$ are equivalent. Finally, all over the text,  $C$, $C_k$, $C_{k,j}$ will
denote positive constants depending only on the displayed
parameters but not necessarily the same at distinct
occurrences. The same remark holds for lower case letters.

\section{Preliminary results}

In this section, we give some fundamental facts about symmetric cones, Berezin transform and related results.

\subsection{Symmetric cones, Bergman metric and estimations of the determinant function}
It is well known that a symmetric cone $\Omega$ induces in $V\equiv \mathbb R^n$
a structure of Euclidean Jordan algebra, in which $\overline
{\Omega}=\{x^2:x\in V\}$. Let ${\bf e}$ be the identity element in
$V$. Denote by $G(\Omega)$ the group of transformations of $\mathbb R^n$ leaving invariant  $\Omega$. We recall that  the group $G(\Omega)$ acts transitively on $\Omega.$ We denote by $H$ the subgroup of $G(\Omega)$ that acts simply transitively on $\Omega$, that is for $x,y\in\Omega$ there is a unique $h\in H$ such that $y=hx.$ Observe that if we still denote by $\mathbb R^n$ the group of translation by vectors in $\mathbb R^n$, then the group $G(\mathcal D)=\mathbb {R}^n\times H$ acts simply transitively on $\mathcal D$.

Recall that $\delta>0$,  $$B_\delta(z)=\{w\in \mathcal {D}: d(z,w)<\delta\}$$ is the Bergman ball centered at $z$ with radius $\delta$, where $d(\cdot,\cdot)$ is the Bergman distance (for a definition, see for example \cite{NS}).  It well known that the
measure $d\lambda(z)=\Delta^{-2n/r}(\Im z)dV(z)$ is an
invariant measure on $\mathcal D$ under the actions of $G(\mathcal D)=\mathbb {R}^n\times H$.

We recall the following (see \cite[Theorem 5.4]{BBGNPR}).
\begin{lem}\label{lem:covering}
For any $\delta\in (0, 1)$, there exists a sequence $\{\zeta_j\}$ of points  of $\mathcal D$ called
$\delta$-lattice such that
then
\begin{itemize}
\item[(i)] the balls $B_j'=B_{\frac{\delta}{2}}(\zeta_j)$ are pairwise disjoint;
\item[(ii)] the balls $B_j=B_\delta(\zeta_j)$ cover $\mathcal D$. Moreover, there is an integer N (depending only on $\mathcal{D}$) such
that each point of $\mathcal D$ belongs to at most N of these balls.
\end{itemize}
\end{lem}

\begin{rem}\label{rem:separated}
Let $A>0$ be fixed. Then any $\delta$-lattice $\{\zeta_j\}$ admits a decomposition into a finite number of  sequences $\{\zeta_{j_k}\}$ satisfying $d(\zeta_{j_{k_1}}, \zeta_{j_{k_2}})\geq A$ for $j_{k_1}\neq j_{k_2}.$
\end{rem}

We observe that
$$\int_{B_j}dV_\nu(z)\approx \int_{B'_j}dV_\nu(z)\approx
C_{\delta}\Delta^{\nu+n/r}(\Im \zeta_j).$$

\vskip .2cm
We refer to \cite[Theorem 5.6]{BBGNPR} for the following sampling theorem.
\begin{lem}\label{lem:sampling}
Let $\{\zeta_{j}\}_{j\in \N}$ be a $\da$-lattice in $\mathcal D$,
$\da \in (0,1)$. Then the following
assertions hold.
\begin{itemize}
\item[(1)] There is a positive
constant $C_{\da}$ such that every $f\in A_{\nu}^{p}(\mathcal D)$ satisfies
$$ ||\{f(\zeta_{j})\Delta^{\frac{1}{p}(\nu+\frac nr)}(\Im \zeta_j)\}||_{l^p} \le C_{\da} ||f||_{p,\nu}. $$
\item[(2)] Conversely, if $\da$ is small enough, there is a
positive constant $C_{\da}$ such that every $f\in A_{\nu}^{p}(\mathcal D)$
satisfies $$||f||_{p,\nu} \le
C_{\da}||\{f(\zeta_{j})\Delta^{\frac{1}{p}(\nu+\frac nr)}(\Im \zeta_j)\}||_{l^p}.$$
\end{itemize}
\end{lem}

We have the following atomic decomposition with change of weight which is derived from \cite[Theorem 3.2]{NS}.
\begin{thm}\label{theo:atomdecompo} Let  $\mu,\,\,\nu >\frac{n}{r}-1$.  Assume that
the operator $P_{\mu}$ is bounded on $L_{\nu}^{2}(\mathcal D)$ and let $\{\zeta_{j}\}_{j\in \N}$
be a $\da$-lattice in $\mathcal D$. Then the following assertions hold.
\begin{itemize}
\item[(i)] For every complex sequence $\{\la_{j}\}_{j\in \N}$ in
$l^{2}$, the series  $$\sum_{j} {\lambda_{j}K_{\mu}(z,
\zeta_{j})\Delta^{\mu + \frac{n}{r}-\frac{1}{2}(\nu+\frac nr)}(\Im \zeta_{j})}$$ is convergent in
$A_{\nu}^{2}(\mathcal D)$. Moreover, its sum $f$ satisfies the inequality
$$||f||_{2,\nu}\leq C_{\delta}||\{\la_{j}\}||_{l^{2}},$$ where
$C_{\delta}$ is a positive constant.
\item[(ii)] For $\da$ small
enough, every function $f\in A_{\nu}^{2}(\mathcal D)$ may be written as
$$ f(z) = \sum_{j} {\lambda_{j}K_{\mu}(z,
\zeta_{j})\Delta^{\mu + \frac{n}{r}-\frac{1}{2}(\nu+\frac nr)}(\Im \zeta_{j})}$$ with
\begin{equation}\label{eq:reverseineqatomdecomp}
||\{\la_{j}\}||_{l^{2}}\le
C_{\da}||f||_{2,\nu}\Ee where $C_{\da}$ is a positive
constant.
\end{itemize}
\end{thm}

The following consequence of the mean value theorem (see \cite{BBGNPR}) is needed.
\begin{lem}\label{lem:meanvalue}
There exists a constant $C>0$ such that for any $f\in \mathcal {H}(\mathcal D)$ and $\delta\in (0,1]$, the following holds
\begin{equation}\label{eq:meanvalue}
|f(z)|^p\le C\delta^{-n}\int_{B_\delta(z)}|f(\zeta)|^p\frac{dV(\zeta)}{\Delta^{2n/r}(\Im \zeta)}.
\end{equation}
\end{lem}
We recall the following integrability conditions for the determinant function (see \cite[Lemma 3.20]{BBGNPR}).
\begin{lem}\label{lem:Apfunction} Let $\alpha$ be real. Then
the function
$f(z)=\Delta^{-\alpha}(\frac{z+it}{i})$, with $t \in \O$, belongs to
$L_{\nu}^{p}(\mathcal{D})$ if and only if $\nu>\frac{n}{r}-1$ and $p\alpha >
\nu+2\frac{n}{r}-1$. In this
case,$$||f||_{p,\nu}^p=C_{\alpha,p}\Da^{-p\alpha +\frac{n}{r}
+ \nu}(t).$$
\end{lem}
We will be using the following Kor\'anyi's lemma.
\begin{lem}\label{lem:Koranyi}\cite[Theorem 1.1]{BIN}
For every $\delta>0,$ there is a constant $C_{\delta}>0$ such that
$$\left|\frac{K(\zeta,z)}{K(\zeta,w)}-1\right|\leq C_{\delta}d(z,w)$$
for all $\zeta,z,w\in \mathcal D,$ with $d(z,w)\leq \delta.$
\end{lem}

We close this subsection by recalling the following consequence of \cite[Corollary 3.4]{BBGNPR} and the above Kor\'anyi's lemma.
\begin{lem}\label{kor}
Let $\nu>\frac nr-1,\,\,\da>0$ and $z,w\in \mathcal D.$  
There is a positive constant $C_\da$ such that for all $z\in B_\da(w),$
$$V_\nu(B_\da(w))|k_\nu(z,w)|^2\leq C_\da.$$

If $\da$ is sufficiently small, then there is $C>0$ such that for all $z\in B_\da(w),$
$$V_\nu(B_\da(w))|k_\nu(z,w)|^2\geq (1-C\da).$$
\end{lem}

The following was first proved  \cite[Lemma 5.1]{BT} in the case of homogeneous Siegel domains of type II.
\begin{lem}\label{lem:sumdeltafunctestim}
Let $0<\delta\leq 1$, $\alpha,\beta\in \mathbb{R}$ with $\alpha>2\frac nr-1$, $\beta>2\frac nr-1$ and $\alpha>\beta+\frac nr-1$. Then for any $\varepsilon>0$, there exists $A_\varepsilon>0$ such that if $\{z_j=x_j+iy_j\}$ is a sequence of points of $\mathcal{D}$ in a $\delta$-lattice satisfying $\inf_{j\neq k} d(z_j, z_k)\geq A_\varepsilon,$ then for any integer $j$, the following estimate holds
\begin{eqnarray}
\sum_{\{k: k\neq j\}}|\Delta^{-\alpha}(z_k-\bar z_j)|\Delta^\beta(y_k)\leq \varepsilon \Delta^{-\alpha+\beta}(y_j).
\end{eqnarray}
\end{lem}
\begin{proof}
Let us give ourself $A>0$. Thanks to Remark \ref{rem:separated}, we may assume that  the sequence $\{z_j=x_j+iy_j\}$ is such that $d(z_j, z_k)\ge A$ for all $j\neq k.$  We first observe with Lemma \ref{lem:meanvalue} that
$$|\Delta^{-\alpha}(z_k-\bar z_j)|\leq (\delta/3)^{-n}\int_{B'_k}|\Delta^{-\alpha}(w-\bar z_j)|\frac{dV(w)}{\Delta^{\frac{2n}{r}}(\Im w)}.$$
It follows that 
\Beas
S &:=& \sum_{\{k: k\neq j\}}|\Delta^{-\alpha}(z_k-\bar z_j)|\Delta^\beta(y_k)\\ &\leq& C(\delta/3)^{-n}\sum_{\{k: k\neq j\}}\int_{B'_k}|\Delta^{-\alpha}(w-\bar z_j)|\Delta^\beta(\Im w)\frac{dV(w)}{\Delta^{\frac{2n}{r}}(\Im w)}\\
&\leq&C(\delta/3)^{-n}\int_{\mathcal{B}}|\Delta^{-\alpha}(w-\bar z_j)|\Delta^\beta(\Im w)\frac{dV(w)}{\Delta^{\frac{2n}{r}}(\Im w)}
\Eeas
where $\mathcal{B}:=\bigcup_{k\neq j}B'_k$. 
\vskip .2cm
Now observe that if $w\in \mathcal{B}$, then $w\in B_k'$ for some $k$ and so $$d(w, z_k)<\frac{\delta}{2}<\frac{A}{2}$$ and for $j\neq k$, $$d(w,z_j)\geq d(z_j, z_k)-d(w, z_k)>A-\frac{A}{2}=\frac{A}{2}.$$ Let $g\in G(\mathcal{D})$ be the transformation such that $g(i{\bf e})=z_j$, and put $w=g(\zeta)$. Observe that for any $s\in \mathbb{R}$, $$\Delta^{s}(w-\bar z_j)=\Delta^{s}(g(\zeta+i{\bf e}))=(Det g)^{\frac{r}{n}s}\Delta^{s}(\zeta+i{\bf e})=\Delta^s(\Im z_j)\Delta^{s}(\zeta+i{\bf e}),$$ $$\Delta^s(\Im w)=(Det g)^{\frac{r}{n}s}\Delta^s(\Im \zeta)=\Delta^s(\Im z_j)\Delta^s(\Im \zeta)$$ and $$dV(w)=(Det g)^2dV(\zeta)=\Delta^{2\frac nr}(\Im z_j)dV(\zeta).$$ It follows that
\Beas
S &\leq& C(\delta/3)^{-n}\int_{d(z_j, w)>A/2}|\Delta^{-\alpha}(w-\bar z_j)|\Delta^\beta(\Im w)\frac{dV(w)}{\Delta^{\frac{2n}{r}}(\Im w)}\\
&\leq& C(\delta/3)^{-n}\Delta^{-\alpha+\beta}(y_j)\int_{d(ie, \zeta)>A/2}|\Delta^{-\alpha}(\zeta+ie)|\Delta^\beta(\Im \zeta)\frac{dV(\zeta)}{\Delta^{\frac{2n}{r}}(\Im \zeta)}.
\Eeas

From the assumptions on $\alpha$ and $\beta$ together with Lemma \ref{lem:Apfunction}, one has that the integral
$$\int_{T_\Omega}|\Delta^{-\alpha}(\zeta+ie)|\Delta^\beta(\Im \zeta)\frac{dV(\zeta)}{\Delta^{\frac{2n}{r}}(\zeta)}$$
converges. Hence, there exists $A_\varepsilon>0$ such that for all $A\ge A_\varepsilon,$ the following inequality holds
$$\int_{d(ie, \zeta)>A/2}|\Delta^{-\alpha}(\zeta+ie)|\Delta^\beta(\Im \zeta)\frac{dV(\zeta)}{\Delta^{\frac{2n}{r}}(\Im \zeta)}\leq \frac{\varepsilon}{C(\delta/3)^{-n}}$$
with $C$ as in the above estimate of $S$. The proof is complete.

\end{proof}

\subsection{Averaging functions and Berezin transform}
The following was proved in \cite{NS} for $1\le p\le \infty$. A careful observation of the proof of \cite[Lemma 2.9]{NS} shows that the result extends to $0<p<1$.
\begin{lem}\label{lem:variationoflattice}
Let $0< p\le \infty$, $\nu\in \mathbb R$, and $\delta,\beta\in (0,1)$. Let $\mu$ be a positive Borel measure on $\mathcal D$. Then
the following assertions are equivalent.
\begin{itemize}
\item[(i)] The function  $\mathcal {D}\ni z\mapsto \frac{\mu(B_\delta(z))}{\Delta^{\nu+\frac{n}{r}}(\Im z)}$ belongs to $L_\nu^p(\mathcal D)$.
\item[(ii)] The function  $\mathcal {D}\ni z\mapsto \frac{\mu(B_\beta(z))}{\Delta^{\nu+\frac{n}{r}}(\Im z)}$ belongs to $L_\nu^p(\mathcal D)$.
\end{itemize}
\end{lem}
Note that the above lemma allows flexibility on the choice of the radius of the ball. This fact is quite useful as seen in \cite{NS}.
\vskip .2cm
We have the following result.
\begin{lem}\label{lem:integraldiscretizationAverBer}
Let $0< p\le 1$, $\nu>\frac{n}{r}-1$, $\beta, \delta\in (0,1)$. Let $\{\zeta_j\}_{j\in \mathbb N}$ be a $\delta$-lattice in $\mathcal D$, and let $\hat {\mu}_\beta$ and $\tilde {\mu}$ be in this order, the average function and the Berezin transform associated to the weight $\nu$ . Then the following assertions are equivalent.
\begin{itemize}
\item[(i)] $\hat {\mu}_\beta\in L^p(\mathcal D,d\lambda)$.
\item[(ii)] $\{\hat {\mu}_\delta(\zeta_j)\}_{j\in \mathbb N}\in l^p$.
\end{itemize}
If moreover, $p>\frac{2\frac nr-1}{\nu+\frac nr}$, then the above assertions are equivalent to
\begin{itemize}
\item[(iii)] $\tilde {\mu}\in L^p(\mathcal D, d\lambda)$.
\end{itemize}
\end{lem}
\begin{proof}
The equivalence (i)$\Leftrightarrow$(ii) follows as in \cite[Lemma 2.12]{NS}. That (iii)$\Rightarrow$(i) follows from the fact that for any $\delta\in (0,1)$, there exists a onstant $C_\delta>0$ such that for any $z\in \mathcal{D}$, $$\hat {\mu}_\beta(z)\le C_\delta\tilde {\mu}(z)$$
(see \cite[Lemma 2.8]{NS}).
To finish the proof, let us prove that (ii)$\Rightarrow$(iii). 
First using Lemma \ref{lem:Koranyi}, we obtain
\Beas
\tilde {\mu}(z) &:=& \int_{\mathcal{D}}|K_\nu(z,w)|^2\Delta^{\nu+\frac nr}(z)d\mu(w)\\ &\le& \sum_{k}\int_{B_k}|K_\nu(z,w)|^2\Delta^{\nu+\frac nr}(z)d\mu(w)\\ &\le& C\sum_{k}|K_\nu(z,\zeta_k)|^2\Delta^{\nu+\frac nr}(z)\mu(B_k)\\ &\le& C\sum_{k}|K_\nu(z,\zeta_k)|^2\Delta^{\nu+\frac nr}(z)\Delta^{\nu+\frac nr}(\zeta_k)\hat{\mu}_\delta(\zeta_k).
\Eeas
As $0<p\le 1$, it follows that
$$\left(\tilde {\mu}(z)\right)^p\le C\sum_{k}|K_\nu(z,\zeta_k)|^{2p}\Delta^{p(\nu+\frac nr)}(z)\Delta^{p(\nu+\frac nr)}(\zeta_k)\left(\hat{\mu}_\delta(\zeta_k)\right)^p.$$
Hence using that $p>\frac{2\frac nr-1}{\nu+\frac nr}$ together with Lemma \ref{lem:Apfunction}, we obtain
\Beas
L &:=& \int_{\mathcal{D}}\left(\tilde {\mu}(z)\right)^pd\lambda(z)\\ &\le& C\sum_{k}\Delta^{p(\nu+\frac nr)}(\zeta_k)\left(\hat{\mu}_\delta(\zeta_k)\right)^p\int_{\mathcal{D}}|K_\nu(z,\zeta_k)|^{2p}\Delta^{p(\nu+\frac nr)-2\frac nr}(z)dV(z)\\ &\le& C\sum_{k}\left(\hat{\mu}_\delta(\zeta_k)\right)^p<\infty.
\Eeas
The proof is complete.
\end{proof}

\subsection{Schatten class operators}
In this subsection,  $\mathcal{H}$ is a Hilbert
space with associated norm $\|\cdot\|$. The
spaces of bounded and compact linear operators on
$\mathcal{H}$ are denoted $\mathcal{B}(\mathcal{H})$ and
$\mathcal{K}(\mathcal{H})$ respectively. We recall that any positive operator $T\in
\mathcal{K}(\mathcal{H})$, then one can find an orthonormal set
$\{e_j\}$ of $\mathcal{H}$ and a sequence $\{\lambda_j\}$
that decreases to $0$ such that
\begin{equation}
Tf=\sum_{j=0}^{\infty}\lambda_j\langle f,e_j\rangle e_j,\,\,\,\,\,\,\,\,\,\,
f\in \mathcal{H} .\end{equation}

For $0< p <\infty$, we say a
compact operator $T$ with a decomposition as above belongs to
the Schatten-Von Neumann p-class
$\mathcal{S}_p:=\mathcal{S}_p(\mathcal{H})$, if 

$$||T||_{\mathcal{S}_p}:=(\sum_{j=0}^{\infty}|\lambda_j|^p)^{\frac{1}{p}}<\infty.$$

For $p=1$, we denote by $\mathcal{S}_1=\mathcal{S}_1(\mathcal{H})$ 
the trace class. We recall that for $T\in \mathcal{S}_1$, the trace of
$T$ is defined by $$Tr(T)=\sum_{j=0}^{\infty}\langle Te_j,e_j\rangle$$
where $\{e_j\}$ is any orthonormal basis of the Hilbert
space $\mathcal{H}$. 

It is known that a compact operator T on $\mathcal H$ belongs
to the Schatten class $\mathcal {S}_p$ if and only if the positive operator $ (T^*T)^{1/2}$ belongs
to $\mathcal {S}_p$, where $T^*$ denotes the adjoint of $T$. In this case, we have $||T||_{\mathcal{S}_p}=||(T^*T)^{1/2}||_{\mathcal{S}_p}$. It is also well known that a positive $T$ belongs to $\mathcal {S}_p$ if and only if the operator $T^p$ belongs to the trace class $\mathcal {S}_1$. In this case, $||T||_{\mathcal{S}_p}=||T^p||_{\mathcal{S}_1}$.
\vskip .2cm
We also recall that if $T$ is a compact operator on $\mathcal H$, and $p\ge 1$, then that $T\in \mathcal S_p$ is equivalent to
$$\sum_{j}|\langle Te_j,e_j\rangle|^p<\infty$$ for any orthonormal set $\{e_j\}$ in $\mathcal H$ (see \cite{Zhu1}).
\vskip .2cm

The following can be found in \cite{Zhu1}.
\begin{lem}\label{lem:schattensmallexpo}
Suppose that $T$ is a positive operator on $\mathcal{H}$, and that $\{e_j\}$ is an orthonormal basis on $\mathcal{H}$. Then if $0<p<1$ and $$\sum_{j=1}^\infty \langle Te_j,e_j\rangle^p<\infty,$$
then $T$ belongs to $\mathcal {S}_p$.
\end{lem}
We also observe the following (see \cite{Zhu2})
\begin{lem}\label{lem:schattenviaorthoper}
Let $T$ be any bounded operator on $\mathcal{H}$ and assume that $A$ is bounded surjective operator on $\mathcal{H}$. Then $T$ belongs to $\mathcal {S}_p$ if and only if the operator $A^*TA$ belongs to $\mathcal {S}_p$.
\end{lem}
Finally, we will need the following result (see \cite{Luecking2})
\begin{lem}\label{lem:schattsufflower}
Let $T$ be any bounded operator on $\mathcal{H}$ and let $\{e_k\}$ be an orthonormal basis of $\mathcal{H}$. Then for any $0<p\le 2$, we have $$\|T\|_{\mathcal{S}_p}^p\le \sum_{k}\sum_j|\langle Te_k,e_j\rangle|^p.$$
\end{lem}

\section{Schatten class membership of Toeplitz operators}
The aim of this section is to give criteria for Schatten class
membership of Toeplitz  operators on the weighted Bergman space
$A_\nu^2(\mathcal D)$.

\subsection{Proof of Theorem \ref{theo:main1}}
We start by proving the following.
\begin{lem}\label{theo:main11}
Let $\mu$ be a positive Borel measure on $\mathcal D$, and $\nu>\frac{n}{r}-1$. Assume that $\frac{\frac nr-1}{\nu+\frac nr}<p<1$. Suppose that for any $\delta$-lattice ($\delta\in (0,1)$) $\{\zeta_j\}_{j\in \mathbb N}$ in the Bergman metric of $\mathcal D$, the sequence $\{\hat {\mu}_\delta(\zeta_j)\}$ belongs to $l^p$, that is \begin{equation*}
\sum_j\left(\frac{\mu(B_j)}{\Da^{\nu+n/r}(\Im \zeta_j)}\right)^p<\infty.
\end{equation*}
Then the Toeplitz operator $T_\mu$ belongs
to the Schatten class $\mathcal {S}_p(A_\nu^2(\mathcal D))$. Moreover, $$\|T_\mu\|_{\mathcal{S}_p}^p\lesssim \sum_j\left(\frac{\mu(B_j)}{\Da^{\nu+n/r}(\Im \zeta_j)}\right)^p.$$
\end{lem}
\begin{proof} 
Let $\sigma$ be large enough so that $P_\sigma$ is bounded on $L_\nu^2(\mathcal{D})$. Thanks to Lemma \ref{lem:variationoflattice}, we can suppose that $\delta$ is small enough so that any $f\in A_\nu^2(\mathcal{D})$ can represented as in Theorem \ref{theo:atomdecompo}. That is $$ f(z) = \sum_{j} {\lambda_{j}K_{\sigma}(z,
\zeta_{j})\Delta^{\sigma + \frac{n}{r}-\frac{1}{2}(\nu+\frac nr)}(\Im \zeta_{j})}$$ with
$
||\{\la_{j}\}||_{l^{2}}\asymp
||f||_{2,\nu}.$
\vskip .2cm
Let $\{e_k\}_{k\ge 1}$ be a fixed orthonormal basis on $A_\nu^2(\mathcal{D})$. Consider the operator $S:A_\nu^2(\mathcal{D})\rightarrow A_\nu^2(\mathcal{D})$ defined by $$S(e_k)=f_k$$
where $$f_k(z)=K_{\sigma}(z,
\zeta_{k})\Delta^{\sigma + \frac{n}{r}-\frac{1}{2}(\nu+\frac nr)}(\Im \zeta_{k}).$$
Then it follows from Theorem \ref{theo:atomdecompo} that $S$ is a bounded and surjective operator on $A_\nu^2(\mathcal{D})$. We know from Lemma \ref{lem:schattenviaorthoper} that $T_\mu$ belongs to $\mathcal {S}_p(A_\nu^2(\mathcal D))$ if and only if $T=S^*T_\mu S$ belongs to $\mathcal {S}_p(A_\nu^2(\mathcal D))$.  It follows from Lemma \ref{lem:schattensmallexpo} that we only have to prove that $$L:=\sum_{k=1}^\infty\langle Te_k,e_k\rangle^p<\infty.$$
We first observe that
$$\langle Te_k,e_k\rangle=\langle T_\mu f_k,f_k\rangle=\int_{\mathcal{D}}|f_k(z)|^2d\mu(z).$$
Hence using Lemma \ref{kor}, we obtain
\Beas
\langle Te_k,e_k\rangle &\le&  \sum_{j=1}^\infty \int_{B_\delta(\zeta_j)}|f_k(z)|^2d\mu(z)\\ &\le& C\sum_{j=1}^\infty |f_k(\zeta_j)|^2\mu(B_\delta(\zeta_j)).
\Eeas
Recalling that $0<p<1$, we then obtain
\Beas
\langle Te_k,e_k\rangle^p &\le& C\sum_{j=1}^\infty |f_k(\zeta_j)|^{2p}(\mu(B_\delta(\zeta_j)))^p\\ &\asymp& \sum_{j=1}^\infty |f_k(\zeta_j)|^{2p}\Delta^{\nu+\frac nr}(\Im \zeta_j)(\hat{\mu}_\delta(\zeta_j))^p.
\Eeas
Thus
\Beas
L &:=& \sum_{k=1}^\infty\langle Te_k,e_k\rangle^p\\ &\le& C\sum_{k=1}^\infty\sum_{j=1}^\infty |f_k(\zeta_j)|^{2p}\Delta^{\nu+\frac nr}(\Im \zeta_j)(\hat{\mu}_\delta(\zeta_j))^p\\ &\le& C\sum_{j=1}^\infty \Delta^{p(\nu+\frac nr)}(\Im \zeta_j)(\hat{\mu}_\delta(\zeta_j))^p\sum_{k=1}^\infty |f_k(\zeta_j)|^{2p}.
\Eeas
Using the fact that each point in $\mathcal{D}$ belongs to at most $N$ balls $B_k$ and the condition $\frac{\frac nr-1}{\nu+\frac nr}<p<1$, we obtain using Lemma \ref{lem:Apfunction}, the following for the inner sum
\Beas
L_j &:=& \sum_{k=1}^\infty |f_k(\zeta_j)|^{2p}\\ &=& \sum_{k=1}^\infty |K_{\sigma}(\zeta_j,
\zeta_{k})|^{2p}\Delta^{2p(\sigma + \frac{n}{r}-\frac{1}{2}(\nu+\frac nr))}(\Im \zeta_{k})\\ &\le& C\sum_{k=1}^\infty\int_{B_\delta(\zeta_k)} |K_{\sigma}(\zeta_j,
z)|^{2p}\Delta^{2p(\sigma + \frac{n}{r}-\frac{1}{2}(\nu+\frac nr))}(\Im z)dV(z)\\ &\le& CN\int_{\mathcal{D}}|K_{\sigma}(\zeta_j,
z)|^{2p}\Delta^{2p(\sigma + \frac{n}{r}-\frac{1}{2}(\nu+\frac nr))}(\Im z)dV(z)\\ &\le& CN\Delta^{-p(\nu+\frac nr)}(\Im \zeta_j).
\Eeas
Using the latter, we conclude that
\Beas
L &:=& \sum_{k=1}^\infty\langle Te_k,e_k\rangle^p\\ &\le& C\sum_{j=1}^\infty(\hat{\mu}_\delta(\zeta_j))^p<\infty.
\Eeas
\end{proof}
We next prove the reverse of the above result.
\begin{lem}\label{theo:main12}
Let $\mu$ be a positive measure on $\mathcal{D}$. Assume that $T_\mu\in \mathcal{S}_p(A_\nu^2(\mathcal{D}))$ for some $0<p<1$. Let $\{\zeta_j\}_{j\in \mathbb{N}}$ be a $\delta$-lattice in $\mathcal{D}$. Then the sequence $\{\hat{\mu}_\delta(\zeta_j)\}$ belongs to $l^p$. Moreover, $$\sum_{j}\left(\hat{\mu}_\delta(\zeta_j)\right)^p\lesssim \|T_\mu\|_{\mathcal{S}_p}^p.$$
\end{lem}
\begin{proof}
We start by considering $\sigma$ large enough so that $\sigma+\frac nr$ and $\sigma+\frac nr-\frac{1}{2}(\nu+\frac{n}{r})$ satisfy the conditions in Lemma \ref{lem:sumdeltafunctestim}. Let $\varepsilon>0$, and let $A_\varepsilon$ be as in Lemma \ref{lem:sumdeltafunctestim}. Following Remark \ref{rem:separated}, we may assume that our sequence $\{\zeta_j\}$ is such that $d(\zeta_j,\zeta_k)>A_\varepsilon$ for $j\neq k$. We further assume that $A_\varepsilon$ is large enough so that corresponding balls $B_k$ are disjoint. Consider the following measure: $$d\omega(z)=\sum_{k}\chi_{B_k}(z)d\mu(z).$$
Then $0\le \omega\le \mu$, $\omega=\mu$ on each ball $B_k$. We also have the inequality $\|T_\omega\|_{\mathcal{S}_p}^p\le \|T_\mu\|_{\mathcal{S}_p}^p$.
\vskip .2cm
Now as in the proof of the previous result, we fix an orthonormal basis $\{e_k\}$ of $A_\nu^2(\mathcal{D})$ and consider the same operator $S$ defined on $A_\nu^2(\mathcal{D})$ by $S(e_k)=f_k$ with $$f_k(z)=K_\sigma(z,\zeta_k)\Delta^{\sigma+\frac nr-\frac{1}{2}(\nu+\frac{n}{r})}(\Im \zeta_j).$$
We recall with Theorem \ref{theo:atomdecompo} that $S$ is bounded and surjective on $A_\nu^2(\mathcal{D})$. Put again $T=S^*T_\omega S$. Then as $T_\omega\in \mathcal{S}_p(A_\nu^2(\mathcal{D}))$, $T$ also belongs to $ \mathcal{S}_p(A_\nu^2(\mathcal{D}))$ and we have $$\|T\|_{\mathcal{S}_p}\le \|S\|^2\|T_\omega\|_{\mathcal{S}_p}\le C\|T_\mu\|_{\mathcal{S}_p}.$$
The main idea of the proof is to show that the $l^p$-norm of the sequence $\{\hat{\mu}_{\delta}(\zeta_j)\}$ is up to a constant a lower bound for $\|T\|_{\mathcal{S}_p}$. For this we decompose $T$ as $T=D+R$, where $D$ is the positive diagonal operator on $A_\nu^2(\mathcal{D})$ given by $$Df:=\sum_{k}\langle Te_k,e_k\rangle\langle f,e_k\rangle e_k,\,\,\,f\in A_\nu^2(\mathcal{D})$$
and $R=T-D$.
We observe that $$\|T\|_{\mathcal{S}_p}^p\ge \frac{1}{2}\|D\|_{\mathcal{S}_p}^p-\|R\|_{\mathcal{S}_p}^p.$$
Hence, if we can prove that
$$\|D\|_{\mathcal{S}_p}^p\ge c_1\sum_j\left(\hat{\mu}_{\delta}(\zeta_j)\right)^p$$
and 
$$\|R\|_{\mathcal{S}_p}^p\le c_2\sum_j\left(\hat{\mu}_{\delta}(\zeta_j)\right)^p$$
with $c_2$ as small as we want, then the proof will be completed.
\vskip .1cm
We start by estimating the diagonal operator $D$. As $D$ is positive, we have
\Beas
\|D\|_{\mathcal{S}_p}^p &=& \sum_k\langle Te_k,e_k\rangle^p=\sum_k\langle T\omega f_k,f_k\rangle^p\\ &=& \sum_k\left(\int_{\mathcal{D}}|f_k(z)|^2d\omega(z)\right)^p\\ &\ge& \sum_k\left(\int_{B_k}|f_k(z)|^2d\omega(z)\right)^p\\ &=& \sum_k\left(\int_{B_k}|f_k(z)|^2d\omega(z)\right)^p\\ &=& \sum_k\left(\int_{B_k}|f_k(z)|^2d\mu(z)\right)^p\\ &\asymp& \sum_{k}\left(\hat{\mu}_\delta(\zeta_k)\right)^p.
\Eeas
That is 
$$\|D\|_{\mathcal{S}_p}^p\ge c_1\sum_{j}\left(\hat{\mu}_\delta(\zeta_j)\right)^p.$$
We now turn to the estimation of $\|R\|_{\mathcal{S}_p}^p$.  First, using Lemma \ref{lem:schattensmallexpo}, we obtain
\Beas
\|R\|_{\mathcal{S}_p}^p &\le& \sum_k\sum_j|\langle Re_j,e_k\rangle|^p\\ &=& \sum_{\{j,k:j\neq k\}}|\langle T_\omega f_j,f_k\rangle|^p\\ &=& \sum_{\{j,k:j\neq k\}}\left|\int_{\mathcal{D}}f_j(z)f_k(z)d\omega(z)\right|^p\\ &\le& \sum_{\{j,k:j\neq k\}}\left(\int_{\mathcal{D}}|f_j(z)||f_k(z)|d\omega(z)\right)^p.
\Eeas 
As the balls $B_l$ are disjoint, using Lemma \ref{kor}, we obtain
\Beas
\int_{\mathcal{D}}|f_j(z)||f_k(z)|d\omega(z) &=& \sum_{l}\int_{B_l}|f_j(z)||f_k(z)|d\mu(z)\\ &\le& C\sum_i|f_j(\zeta_l)||f_k(\zeta_l)|\mu(B_l)\\ &\asymp& \sum_l|f_j(\zeta_l)||f_k(\zeta_l)|\Delta^{\nu+\frac nr}(\Im \zeta_l)\hat{\mu}_\delta(\zeta_l).
\Eeas
As $0<p<1$, it follows that
\Beas
\|R\|_{\mathcal{S}_p}^p &\le& C\sum_l\Delta^{p(\nu+\frac nr)}(\Im \zeta_l)\left(\hat{\mu}_\delta(\zeta_l)\right)^pL_l
\Eeas
where
\Beas
L_l &:=& \sum_{\{j,k:j\neq k\}}|f_j(\zeta_l)|^p|f_k(\zeta_l)|^p\\ &=& \sum_{\{j,k:j\neq k\}}|\Delta^{-p(\mu+\frac nr)}(\frac{\zeta_j-\overline{\zeta}_l}{i})|\Delta^{p(\mu+\frac nr-\frac{1}{2}(\nu+\frac nr))}(\Im \zeta_j)\times\\ & & |\Delta^{-p(\mu+\frac nr)}(\frac{\zeta_k-\overline{\zeta}_l}{i})|\Delta^{p(\mu+\frac nr-\frac{1}{2}(\nu+\frac nr))}(\Im \zeta_k)\\ &=& 2L_l^1+L_l^2
\Eeas
with
$$L_l^1:=\Delta^{-\frac{p}{2}(\nu+\frac nr))}(\Im \zeta_l)\sum_{\{j:j\neq l\}}|\Delta^{-p(\mu+\frac nr)}(\frac{\zeta_j-\overline{\zeta}_l}{i})|\Delta^{p(\mu+\frac nr-\frac{1}{2}(\nu+\frac nr))}(\Im \zeta_j)$$
and 
\Beas
L_l^2 &:=& \sum_{\{j,k:j\neq l\neq k\,\,\,\textrm{and}\,\,\,j\neq k\}}|\Delta^{-p(\mu+\frac nr)}(\frac{\zeta_j-\overline{\zeta}_l}{i})|\Delta^{p(\mu+\frac nr-\frac{1}{2}(\nu+\frac nr))}(\Im \zeta_j)\times\\ & & |\Delta^{-p(\mu+\frac nr)}(\frac{\zeta_k-\overline{\zeta}_l}{i})|\Delta^{p(\mu+\frac nr-\frac{1}{2}(\nu+\frac nr))}(\Im \zeta_k).
\Eeas
Using Lemma \ref{lem:sumdeltafunctestim} we obtain
$$L_l^1\le \varepsilon\Delta^{-p(\nu+\frac nr))}(\Im \zeta_l)$$
and 
$$L_l^2\le \varepsilon^2\Delta^{-p(\nu+\frac nr))}(\Im \zeta_l).$$
It follows that
$$\|R\|_{\mathcal{S}_p}^p\le c_2(2\varepsilon+\varepsilon^2)\sum_j\left(\hat{\mu}_\delta(\zeta_j)\right)^p.$$
Hence
$$\|T\|_{\mathcal{S}_p}^p\ge \left[\frac{c_1}{2}-c_2(2\varepsilon+\varepsilon^2)\right]\sum_j\left(\hat{\mu}_\delta(\zeta_j)\right)^p.$$
Taking $\varepsilon$ small enough so that $\frac{c_1}{2}-c_2(2\varepsilon+\varepsilon^2)>0$, we conclude that $$\sum_j\left(\hat{\mu}_\delta(\zeta_j)\right)^p<\infty.$$ The proof is complete.
\end{proof}
We can now prove Theorem \ref{theo:main1}.
\begin{proof}[Proof of Theorem \ref{theo:main1}]
We start by proving the necessity of the condition $p>\frac{2\frac{n}{r}-1}{\nu+\frac nr}$ in assertion (iv). We recall that ${\bf e}$ is the identity element of $V$. We may suppose that $\mu(B_1(i{\bf e}))>0$ (if not change the radius of the Bergman ball). Then using Lemma \ref{lem:Koranyi}, we obtain
\Beas
\tilde{\mu}(z) &=& \int_{\mathcal{D}}|k_z^\nu(w)|^2d\mu(w)\\ &\ge& \int_{B_1(i{\bf e})}|k_z^\nu(w)|^2d\mu(w)\\ &\ge& C\mu(B_1(i{\bf e}))|\Delta^{-(\nu+\frac nr)}(\frac zi+{\bf e})|^2\Delta^{\nu+\frac nr}(\Im z).
\Eeas
It follows that if $\tilde{\mu}(z)\in L^p(\mathcal{D},d\lambda)$, then we should have
$$\int_{\mathcal{D}}|\Delta^{-(\nu+\frac nr)}(\frac zi+{\bf e})|^{2p}\Delta^{p(\nu+\frac nr)}(\Im z)\frac{dV(z)}{\Delta^{2\frac nr}(\Im z)}<\infty$$
which by Lemma \ref{lem:Apfunction} is possible only if $p(\nu+\frac nr)>2\frac nr-1$.
\vskip .2cm
Now the equivalences (ii)$\Leftrightarrow$(iii)$\Leftrightarrow$(iv) are from Lemma \ref{lem:integraldiscretizationAverBer}. The equivalence (i)$\Leftrightarrow$(ii) is derived from Lemma \ref{theo:main11} and Lemma \ref{theo:main12}. The proof is complete.

\end{proof}
\subsection{Schatten class for general operators}
We consider here Schatten class criteria for an arbitrary operator defined on
$A_\nu^2(\mathcal{D})$ with values in a Hilbert space $\mathcal{H}$. We denote by $\mathcal{B}(A^2(\mathcal{D}), \mathcal{H})$ the set of bounded operators from $A_\nu^2(\mathcal{D})$ to $\mathcal{H}$. To avoid any confusion, we denote by $\langle \cdot,\cdot\rangle_\mathcal{H}$ and $\langle\cdot,\cdot\rangle_\nu$ the inner products in $\mathcal{H}$ and $A_\nu^2(\mathcal{D})$ respectively. We start with the Hilbert-Schmidt class $\mathcal{S}_2:=\mathcal{S}_2(A_\nu^2(\mathcal{D}),\mathcal{H})$.
\begin{prop}\label{prop:hilberysmigene} Let $T\in \mathcal{B}(A^2(\mathcal{D}),
\mathcal{H})$ then
$$||T||_{\mathcal{S}_2(A_\nu^2(\mathcal{D}),\mathcal{H})}^2=
C_{n,m}\int_{\mathcal{D}}||T(k_z^{\nu,m})||_{\mathcal{H}}^2d\lambda(z),$$
for every integer $m\ge 0$.  
\end{prop}
\begin{proof} 
This result was proved in \cite{S}. As the definition of Bergman spaces here is quite different, let us give a proof here for completeness.
Let $\{e_j\}$ is an orthonormal basis of $\mathcal{H}$, then
\begin{eqnarray*}
\int_{\mathcal{D}}||T(K_z^{\nu,m})||_{\mathcal{H}}^2
\Delta^{2m+\nu}(\Im z)dV(z) &=&
\int_{\mathcal{D}}\sum_{j=0}^{\infty}
|\langle TK_z^{\nu,m},e_j\rangle_\mathcal{H}|^2 \Delta^{2m+\nu}(\Im z)dV(z)\\
&=& \sum_{j=0}^{\infty}\int_{\mathcal{D}}
|\langle K_z^{\nu,m},T^*e_j\rangle_\nu|^2\Delta^{2m+\nu}(\Im z)dV(z)\\
&=&
\sum_{j=0}^{\infty}\int_{\mathcal{D}}|\Box_z^{m}T^*e_j(z)|^2\Delta^{2m+\nu}(\Im z)dV(z)\\
&=& C_{n,m}\sum_{j=0}^{\infty}\int_{\mathcal{D}}|T^*e_j(z)|^2dV_\nu(z)\\
&=& C_{n,m}\sum_{j=0}^{\infty}||T^*e_j||_{A_\nu^2}^2\\
&=&
C_{n,m}||T^*||_{\mathcal{S}_2}^2=C_{n,m}||T||_{\mathcal{S}_2}^2.\end{eqnarray*}
In the fourth equality, we used the fact that $\Box_z^{m}$
is an isometric (up to constant $C_{n,m}$) isomorphism from
$A_\nu^2(\mathcal{D})$ onto $A_{2m+\nu}^2(\mathcal{D})$.\end{proof}
We will deduce some results from the above one. The first one is the following which follows as in \cite[Lemma 3.2]{S}.
\begin{prop}\label{prop:genebox} Suppose that $T\in
\mathcal{B}(A_\nu^2(\mathcal{D}),\mathcal{H})$. Let $m\ge 0$ be an integer. Then
\begin{itemize}\item[i)] if $T\in \mathcal{S}_p(A_\nu^2(\mathcal{D}),\mathcal{H})$ for $2 < p<\infty$, then
$$\int_{\mathcal{D}}||T(k_z^{\nu,m})||_{\mathcal{H}}^pd\lambda(z)\le C_{n,m}||T||_{\mathcal{S}_p(A_\nu^2(\mathcal{D}),\mathcal{H})}^p.$$
\item[ii)] If for $0< p<2$,
$$\int_{\mathcal{D}}||T(k_z^{\nu,m})||_{\mathcal{H}}^pd\lambda(z) <\infty,$$
then $T\in \mathcal{S}_p(A_\nu^2(\mathcal{D}),\mathcal{H})$.
Moreover,$$||T||_{\mathcal{S}_p(A_\nu^2(\mathcal{D}),\mathcal{H})}^p\le
C_{n,m}\int_{\mathcal{D}}||T(k_z^{\nu,m})||_{\mathcal{H}}^pd\lambda(z).$$\end{itemize}
\end{prop}
We next have the following which is in fact implicit in the proof of the above result in \cite{S}.
\begin{prop}\label{prop:boxberezin} Suppose that $T\in
\mathcal{B}(A_\nu^2(\mathcal{D}))$ is a positive operator. Let $m\ge 0$ be an integer. Then,
\begin{itemize}\item[i)] if $T\in \mathcal{S}_p(A_\nu^2(\mathcal{D}))$ for $1\le p<\infty$, then
$$\int_{\mathcal{D}}|\langle T(k_z^{\nu,m}),k_z^{\nu,m}\rangle_\nu |^pd\lambda(z)\le C_{n,m}||T||_{\mathcal{S}_p(A_\nu^2(\mathcal{D}))}^p.$$
\item[ii)] If for $0< p\le 1$,
$$\int_{\mathcal{D}}|\langle T(k_z^{\nu,m}),k_z^{\nu,m}\rangle_\nu |^pd\lambda(z) <\infty,$$
then $T\in \mathcal{S}_p(A_\nu^2(\mathcal{D}))$.
Moreover,$$||T||_{\mathcal{S}_p(A_\nu^2(\mathcal{D}))}^p\le
C_{n,m}\int_{\mathcal{D}}|\langle T(k_z^{\nu,m}),k_z^{\nu,m}\rangle_\nu |^pd\lambda(z).$$\end{itemize}
\end{prop}

\begin{proof}
From the proof of Proposition \ref{prop:hilberysmigene} , we have that if
$T\in \mathcal{S}_1(A_\nu^2(\mathcal{D}))$ is a
positive  operator, then
$$Tr(T)=\|T^{1/2}\|_{\mathcal{S}_2}^2=C_{n,m}\int_{\mathcal{D}}|\langle T(k_z^{\nu,m}),k_z^{\nu,m}\rangle_\nu |d\lambda(z).$$

Recalling that
$||T||_{\mathcal{S}_p}^p=Tr(T^p)$, the proof follows from the fact that for any unit
vector(see \cite{Zhu1}) in $g\in L^2(\mathcal{D})$, we have
$$\langle T g,g\rangle_\nu^{p}\le \,\,\,\langle T^pg,g\rangle_\nu,\,\,\,\,\, \textrm{if}\,\,\,\,\, p\ge 1$$ and
$$\langle T^pg,g\rangle_\nu\le \,\,\, \langle Tg,g\rangle_\nu^{p}\,\,\,\,\,\, \textrm{if}\,\,\,\,\, 0< p\le 1.$$
\end{proof}
\subsection{Proof of Theorem \ref{thm:main2}}
We start by observing that taking $T=T_\mu$ in Proposition \ref{prop:boxberezin}, we obtain the following reproducing kernel thesis for $T_\mu$.
\begin{cor}\label{cor:boxberezin} Let $\mu$ be a positive measure on $\mathcal{D}$, and let $m\ge 0$ be an integer. Then the  following assertions hold.
\begin{itemize}\item[i)] If $T_\mu\in \mathcal{S}_p(A_\nu^2(\mathcal{D}))$ for $1\le p<\infty$, then
$$\int_{\mathcal{D}}\langle T_\mu(k_z^{\nu,m}),k_z^{\nu,m}\rangle_\nu ^pd\lambda(z)\le C_{n,m}||T_\mu||_{\mathcal{S}_p(A_\nu^2(\mathcal{D}))}^p.$$
\item[ii)] If for $0< p\le 1$,
$$\int_{\mathcal{D}}\langle T_\mu(k_z^{\nu,m}),k_z^{\nu,m}\rangle_\nu ^pd\lambda(z) <\infty,$$
then $T_\mu\in \mathcal{S}_p(A_\nu^2(\mathcal{D}))$.
Moreover,$$||T_\mu||_{\mathcal{S}_p(A_\nu^2(\mathcal{D}))}^p\le
C_{n,m}\int_{\mathcal{D}}\langle T_\mu(k_z^{\nu,m}),k_z^{\nu,m}\rangle_\nu ^pd\lambda(z).$$\end{itemize}
\end{cor}
We next prove the following sufficient condition.
\begin{lem}\label{lem:suffbergrand} Let $1\le p<\infty$, and let $m\ge 0$ be an integer. Assume that $\mu$ is a positive measure on $\mathcal{D}$, . Then if the Toeplitz operator $T_\mu$ satisfies
$$\int_{\mathcal{D}}\langle T_\mu(k_z^{\nu,m}),k_z^{\nu,m}\rangle_\nu^pd\lambda(z)<\infty,$$
then $T_\mu\in \mathcal{S}_p(A_\nu^2(\mathcal{D}))$. Moreover,$$||T_\mu||_{\mathcal{S}_p(A_\nu^2(\mathcal{D}))}^p\le
C_{n,m}\int_{\mathcal{D}}\langle T_\mu(k_z^{\nu,m}),k_z^{\nu,m}\rangle_\nu ^pd\lambda(z).$$

\end{lem}
\begin{proof}
As $p\ge 1$, we only need to prove that there is positive constant $C$ such that for any orthonormal sequence $\{e_k\}$ on $A_\nu^2(\mathcal{D})$,
$$\sum_k\langle T_\mu e_k,e_k\rangle_\nu^p\le C\int_{\mathcal{D}}\langle T_\mu(k_z^{\nu,m}),k_z^{\nu,m}\rangle_\nu^pd\lambda(z).$$
We recall our notation
$$\tilde{\mu}^m(z):=\langle T_\mu k_z^{\nu,m},k_z^{\nu,m}\rangle_\nu=\int_{\mathcal{D}}|k_z^{\nu,m}(w)|^2d\mu(w).$$

We start  by noting that by Lemma \ref{lem:meanvalue}, we have
$$|e_k(z)|^2\le C\delta^{-n}\int_{B_\delta(z)}|e_k(w)|^2d\lambda(w).$$
It follows from this and Lemma \ref{kor} that
$$|e_k(z)|^2\le C\int_{\mathcal{D}}|e_k(w)|^2|k_z^{\nu,m}(w)|^2dV_\nu(w)$$
Hence
\Beas
\langle T_\mu e_k,e_k\rangle_\nu &=& \int_{\mathcal{D}}|e_k(z)|^2d\mu(z)\\ &\le& \int_{\mathcal{D}}|e_k(w)|^2\tilde{\mu}^m(w)dV_\nu(w)
\Eeas
and so using H\"older's inequality, that the $e_k$s are orthonormal and $$\sum_k|e_k(z)|^2\le \|K_z^\nu\|_{2,\nu}^2,$$ we obtain
\Beas
\sum_k\langle T_\mu e_k,e_k\rangle_\nu^p &\le& \sum_k\left(\int_{\mathcal{D}}|e_k(w)|^2\tilde{\mu}^m(w)dV_\nu(w)\right)^p\\ &\le& \int_{\mathcal{D}}\left(\tilde{\mu}^m(w)\right)^p\left(\sum_k|e_k(w)|^2\right)dV_\nu(w)\\  &\le& \int_{\mathcal{D}}\left(\tilde{\mu}^m(w)\right)^p\|K_w^\nu\|_{2,\nu}^2dV_\nu(w)\\ &\lesssim&  \int_{\mathcal{D}}\left(\tilde{\mu}^m(w)\right)^p d\lambda(w).
\Eeas
The proof is complete.
\end{proof}

We now prove the following necessary condition.
\begin{lem}\label{lem:nessberpetit} Let $m\ge 0$ be an integer such that $\max\{\frac{2\frac nr-1}{\nu+\frac nr+2m},\frac{\frac nr-1}{\nu+\frac nr}\}< p\le 1$. Assume that $\mu$ is a positive measure on $\mathcal{D}$. Then if the Toeplitz operator $T_\mu$ belongs to the Schatten class $\mathcal{S}_p(A_\nu^2(\mathcal{D})$, then
$$\int_{\mathcal{D}}\langle T_\mu(k_z^{\nu,m}),k_z^{\nu,m}\rangle_\nu^pd\lambda(z)<\infty.$$
Moreover,$$
\int_{\mathcal{D}}\langle T_\mu(k_z^{\nu,m}),k_z^{\nu,m}\rangle_\nu^pd\lambda(z)\lesssim ||T_\mu||_{\mathcal{S}_p(A_\nu^2(\mathcal{D}))}^p.$$

\end{lem}
\begin{proof}
Assume that the Toeplitz operator $T_\mu$ belongs to the Schatten class $\mathcal{S}_p(A_\nu^2(\mathcal{D})$. Then by Lemma \ref{theo:main11}, this implies that for any $\delta$-lattice $\{\zeta_k\}$ of points of $\mathcal{D}$, the sequence $\{\hat{\mu}_\delta(\zeta_k)\}$ belongs to $l^p$ with  $$\sum_k\left(\hat{\mu}_\delta(\zeta_k)\right)^p\lesssim \|T_\mu\|_{\mathcal{S}_p}^p.$$
It follows that to prove the above lemma, it is enough to prove that there is positive constant $C$ such that for any $\delta$-lattice $\{\zeta_k\}$ of points of $\mathcal{D}$,
$$\int_{\mathcal{D}}\langle T_\mu(k_z^{\nu,m}),k_z^{\nu,m}\rangle_\nu^pd\lambda(z)\le C\sum_k\left(\hat{\mu}_\delta(\zeta_k)\right)^p.$$
Recalling that $0<p<1$ and using Lemma \ref{kor}, we first obtain
\Beas
L &:=& \int_{\mathcal{D}}\left(\tilde{\mu}^m(z)\right)^pd\lambda(z)\\ &=& \int_{\mathcal{D}}\left(\int_{\mathcal{D}}|k_z^{\nu,m}(w)|^2d\mu(w)\right)^pd\lambda(z)\\ &\le& \int_{\mathcal{D}}\left(\sum_k\int_{B_k}|k_z^{\nu,m}(w)|^2d\mu(w)\right)^pd\lambda(z)\\ &\le& C\int_{\mathcal{D}}\left(\sum_k|k_z^{\nu,m}(\zeta_k)|^2\mu(B_k)\right)^pd\lambda(z)\\ &\le& C\int_{\mathcal{D}}\left(\sum_k|K_z^{\nu,m}(\zeta_k)|^2\mu(B_k)\right)^p\frac{\Delta^{p(\nu+2m+\frac nr)}(\Im z)}{\Delta^{2\frac nr}(\Im z)}dV(z)\\ &\le& C\sum_k\left(\mu(B_k)\right)^p\int_{\mathcal{D}}|K_z^{\nu,m}(\zeta_k)|^{2p}\frac{\Delta^{p(\nu+2m+\frac nr)}(\Im z)}{\Delta^{2\frac nr}(\Im z)}dV(z).
\Eeas
The condition on $p$ and Lemma \ref{lem:Apfunction} give us
$$\int_{\mathcal{D}}|K_z^{\nu,m}(\zeta_k)|^{2p}\Delta^{p(\nu+2m+\frac nr)-2\frac nr}dV(z)=C\Delta^{-p(\nu+\frac nr)}(\Im \zeta_k).$$
We then conclude that
\Beas
\int_{\mathcal{D}}\left(\tilde{\mu}^m(z)\right)^pd\lambda(z)&\le& C\sum_k\left(\mu(B_k)\right)^p\Delta^{-p(\nu+\frac nr)}(\Im \zeta_k)\\ &\asymp& \sum_k\left(\hat{\mu}_\delta(\zeta_k)\right)^p.
\Eeas
The proof if complete.
\end{proof}
Theorem \ref{thm:main2} clearly follows from Corollary \ref{cor:boxberezin}, Lemma \ref{lem:suffbergrand} and Lemma \ref{lem:nessberpetit}.

\section{Application to Ces\`aro-type operators}

 We consider the following equivalence class $$\mathcal {N}_n := \{F \in
\mathcal H (\mathcal D) : \Box^n F = 0 \}$$ and set
$$\mathcal{H}_n(\mathcal D)=\mathcal{H}(\mathcal D)/ \mathcal N_n.$$
For $g\in \mathcal {H}(\mathcal D)$, we define the operator $T_g$
 as follows: for $f\in \mathcal {H}(\mathcal D)$, $T_gf$ is the equivalence class of the solutions of the equation
 $$\Box^{n}F=f\Box^{n}g.$$ The operator $T_g$ was called in \cite{NS} Ces\`aro-type operator and it was remarked in the same paper that its definition does not depend on the choice of the representative of
the class of the symbol.

We consider in this part, criteria for Schatten class membership of the Ces\`aro-type operator above on the weighted Bergman space $A_\nu^2(\mathcal D)$. In fact a characterization of Schatten classes for this operator was obtain in \cite{NS} for the range $2\le p\le \infty$. Our aim here is to extend this result to the range $1\le p<2$. We refer to \cite{Constantin, Luecking2, Zhu2, Zhu3} for the corresponding results on some classical domains.

 Let us recall that the Besov space $\mathcal {B}^p(\mathcal D)$ is the subset of $\mathcal {H}_n(\mathcal D)$ consisting of functions $f$ such that
$\Da^n\Box^nf\in L^p(\mathcal {D}, d\lambda)=L^p(\mathcal {D}, \frac{dV(z)}{\Delta^{2\frac{n}{r}}(\Im z)})$. For more on Besov spaces of tube domains over symmetric cones, we refer the reader to \cite{BBGR,BBGRS}.

We now obtain the following.
\begin{thm}\label{theo:mainSchattenCesaro}
Let $1\le p<2 $, $\nu>\frac{n}{r}-1$. If $g$ is a given holomorphic function in $\mathcal D$, then the Ces\`aro-type operator $T_g$ belongs to
$\mathcal {S}_p(A_\nu^2(\mathcal D))$ if and only if $g\in \mathcal {B}^p(\mathcal D)$.
\end{thm}

\begin{proof}
Let us first assume that $T_g\in \mathcal {S}_p(A_\nu^2(\mathcal D))$. Then by Proposition \ref{prop:boxberezin}, we have that
$$\int_{\mathcal{D}}|\langle T_g(k_z^{\nu,n}),k_z^{\nu,n}\rangle_\nu |^pd\lambda(z)\le C||T_g||_{\mathcal{S}_p(A_\nu^2(\mathcal{D}))}^p.$$
\vskip .1cm
Using reproducing formula, we obtain
\Beas
 ||T_g||_{\mathcal{S}_p(A_\nu^2(\mathcal{D}))}^p &\ge& C\int_{\mathcal{D}}|\langle T_g(k_z^{\nu,n}),k_z^{\nu,n}\rangle_\nu |^pd\lambda(z)\\ &=& C_n\int_{\mathcal{D}}|\langle \Box_{\cdot}^n(T_gk_z^{\nu,n}(\cdot)),k_z^{\nu,n}(\cdot)\rangle_{\nu+n} |^pd\lambda(z)\\ &=& C_n\int_{\mathcal{D}}|\langle \Box^ng(\cdot)k_z^{\nu,n}(\cdot),K_z^{\nu,n}(\cdot)\rangle_{\nu+n} |^p\Delta^{\frac p2(\nu+\frac nr+2n)}(z)d\lambda(z)\\ &=& C_n\int_{\mathcal{D}}|\left(\Box_z^n g(z)\right)k_z^{\nu,n}(z)|^p\Delta^{\frac p2(\nu+\frac nr+2n)}(z)d\lambda(z)\\ &=& C_n\int_{\mathcal{D}}|\Delta^n(\Im z)\Box_z^n g(z)|^pd\lambda(z).
\Eeas
Hence $g\in \mathcal{B}^p(\mathcal{D})$ if $T_g\in \mathcal {S}_p(A_\nu^2(\mathcal D))$.
\vskip .3cm
Now assume that $g\in \mathcal{B}^p(\mathcal{D})$. We consider the following measure
 $$d\mu(z)=|\Box^ng(z)|^2\Delta^{2n+\nu-n/r}(\Im z)dV(z).$$ 
 We first observe the following. Let $\{\zeta_k\}$ be a $\delta$-lattice of points of $\mathcal{D}$. Using Lemma \ref{lem:sampling}, we obtain
\Beas
\int_{\mathcal D}|\Delta^n(\Im z)\Box^ng(z)|^pd\lambda(z) &\asymp& 
\sum_j \left(|\Box^ng(\zeta_j)|^2\Delta^{2n}(\Im \zeta_j)\right)^{p/2}\\ &\asymp&
\sum_j\left(\int_{B_j}|\Box^ng(z)|^2\Delta^{2n}(\Im z)\frac{dV(z)}{\Delta^{2n/r}(\Im z)}\right)^{p/2}\\ &\asymp& \sum_j\left(\frac{1}{\Delta^{\nu+n/r}(\Im \zeta_j)}\int_{B_j}d\mu(z)\right)^{p/2}\\ &=& \sum_j\left(\frac{\mu(B_j)}{\Delta^{\nu+n/r}(\Im \zeta_j)}\right)^{p/2}.
\Eeas
That is \Be\label{eq:besovtodistr}\int_{\mathcal D}|\Delta^n(\Im z)\Box^ng(z)|^pd\lambda(z)\asymp \sum_j\left(\frac{\mu(B_j)}{\Delta^{\nu+n/r}(\Im \zeta_j)}\right)^{p/2}.\Ee
We next observe the following. Let $\{e_j\}$ be any orthonormal basis of $A_\nu^2(\mathcal{D})$. Then using H\"older's inequality, we obtain
\Beas
\langle T_ge_j,e_j\rangle_\nu^2 &=& \langle \Box^nT_ge_j,e_j\rangle_{\nu+n}^2\\ &=& \langle \Delta^n\Box^nT_ge_j,e_j\rangle_{\nu}^2\\ &\le& \int_{\mathcal{D}}|e_j(z)|^2d\mu(z)\\ &=& \langle T_\mu e_j,e_j\rangle_\nu.
\Eeas
It follows that \Be\label{eq:cesartotoep}\sum_j|\langle T_ge_j,e_j\rangle_\nu|^p\le \sum_j\langle T_\mu e_j,e_j\rangle_\nu^{p/2}.\Ee
Hence, as $p\ge 1$, to prove that $T_g$ belongs to $\mathcal {S}_p(A_\nu^2(\mathcal D))$, it suffices by (\ref{eq:cesartotoep}) to prove that $$\sum_j\langle T_\mu e_j,e_j\rangle_\nu^{p/2}<\infty.$$
This follows as in the proof of Lemma \ref{theo:main11}, using that as $g\in \mathcal{B}^p(\mathcal{D})$, we have by (\ref{eq:besovtodistr}) that
$$\sum_j\left(\frac{\mu(B_j)}{\Delta^{\nu+n/r}(\Im \zeta_j)}\right)^{p/2}<\infty.$$
The proof is complete.
\end{proof}

\end{document}